\documentclass{amsart}
\usepackage{verbatim}
\usepackage{rotating} %for \includegraphix[angle=270]
\usepackage{amssymb}
\usepackage{mathrsfs,amsfonts,amsmath,amssymb,epsfig,amscd,xy,amsthm,pb-diagram} 
\newtheorem{theorem}{Theorem}[section]
\newtheorem{proposition}[theorem]{Proposition}

\newtheorem{question}[theorem]{Question}

\newtheorem{conjecture}[theorem]{Conjecture}

\theoremstyle{plain}
\theoremstyle{remark}

\newtheorem{example}[theorem]{Example}

\newcommand{\Q}{{\mathbb Q}}

\newcommand{\Z}{{\mathbb Z}}
\newcommand{\N}{{\mathbb N}}

\newcommand{\cE}{{\mathcal E}}

\newcommand{\tensor}{\otimes}

\newcommand{\Gm}{\mathbb{G}_{\text{m}}}

\newcommand{\Qbar}{\bar{\Q}}

\DeclareMathOperator{\Tr}{Tr}

\DeclareMathOperator{\tor}{tor}

\newcommand{\bP}{{\mathbb P}}

\newcommand{\cS}{\mathcal{S}}

\newcommand{\hhat}{{\widehat h}}

\DeclareMathOperator{\chara}{char}
\author{Khoa Dang Nguyen}
\address{
Khoa D.~Nguyen \\
Department of Mathematics and Statistics\\
University of Calgary\\
AB T2N 1N4, Canada
}
\email{dangkhoa.nguyen@ucalgary.ca}


\begin{document}
	\title[The Hermite-Joubert problem]{The Hermite-Joubert problem and a conjecture of Brassil-Reichstein}
	\date{July 2017}
	\begin{abstract}		 
		We show that Hermite theorem fails for every
		integer $n$ of the form $3^{k_1}+3^{k_2}+3^{k_3}$
		with integers $k_1>k_2>k_3\geq 0$. This confirms
		a conjecture of Brassil and Reichstein. We also
		obtain new results for the relative
		Hermite-Joubert problem over a finitely generated 
		field of characteristic $0$.
	\end{abstract}
	
	\maketitle
	
	\section{Introduction} \label{sec:intro}
	The Hermite-Joubert problem in characteristic $0$ is as follows:
	\begin{question}\label{q:Hermite-Joubert}
	Let $n\geq 5$ be an integer. Let $E/F$ be a field extension with $\chara(F)=0$ and $[E:F]=n$, can one
	always find an element $0\neq \delta\in E$ such
	that $\Tr_{E/F}(\delta)=\Tr_{E/F}(\delta^3)=0$?
	\end{question}
	The answer is ``yes'' when $n=5$ and $n=6$ thanks
	to results by Hermite \cite{Her61_Sl} and 
	Joubert \cite{Jou67_Sl} in the 1860s. 
	Modern proofs of these results can be found in
	\cite{Cor87_CH,Kra06_AR}. 
	When
	$n$ has the form $3^k$ for an integer $k\geq 0$ or the form
	$3^{k_1}+3^{k_2}$
	for integers $k_1>k_2\geq 0$, 
	Reichstein \cite{Rei99_OA}
	shows that Question~\ref{q:Hermite-Joubert} has
	the negative answer. The readers are referred to
	\cite{BR97_OT,Rei99_OA,RY02_ED} for further developments 
	and 
	open questions inspired by the Hermite-Joubert 
	problem.
	\emph{This paper is motivated by 
	results and questions in a recent paper by
	Brassil and Reichstein \cite{BR17_TH}
	in which the case $n=3^{k_1}+3^{k_2}+3^{k_3}$
	for integers $k_1>k_2>k_3\geq 0$
	is studied.} Our first main result is the following:
	\begin{theorem}\label{thm:main 1}
	When $n=3^{k_1}+3^{k_2}+3^{k_3}$ for
	integers $k_1>k_2>k_3\geq 0$, Question~\ref{q:Hermite-Joubert} has the negative answer.
	\end{theorem}
	 
	In fact, we will prove a more precise result (see
	Theorem~\ref{thm:precise 1}) answering a conjecture
	of Brassil-Reichstein \cite[Conjecture~14.1]{BR17_TH}.
	As in \cite{BR17_TH}, we can also consider 
	the relative version of Question~\ref{q:Hermite-Joubert} in which $F$ contains a given
	base field $F_0$; in particular, Question~\ref{q:Hermite-Joubert} corresponds to the case $F_0=\Q$.
	Our second result is the following (see 
	Theorem~\ref{thm:precise 2} for a more precise result):
	
	\begin{theorem}\label{thm:main 2}
	Let $F_0$ be a finitely generated field of characteristic $0$. There is a finite subset
	$\cS$ of $\N\times\N$ depending on $F_0$
	such that the following holds.
	For every integer $n$ of the form
	$3^{k_1}+3^{k_2}+3^{k_3}$ for
	integers $k_1>k_2>k_3\geq 0$ with $(k_1-k_3,k_2-k_3)\notin \cS$,
	 Question~\ref{q:Hermite-Joubert} relative to the base field $F_0$ has
	the negative answer.
	\end{theorem}

	{\bf Acknowledgments.}  We wish to thank Professor Zinovy Reichstein for communicating to us his conjecture with Brassil and for many useful discussions.
	
	\section{Proof of Theorem~\ref{thm:main 2}}\label{proof:thm 2}
	Throughout this section, $F_0$ is a finitely
	generated field of characteristic $0$.
	An abelian group $G$ is said to be of finite rank if $\Q\tensor_{\Z}G\cong G/G_{\tor}$ is a finite dimensional vector space over $\Q$. 
	We start with the following result which might be of 
	independent interest:
	\begin{proposition}\label{prop:finite}
		Let $P(Z_1,Z_2,Z_3)\in F_0[Z_1,Z_2,Z_3]$ be a homogeneous polynomial defining a geometrically irreducible plane curve with geometric genus  $g\geq 1$. Let $G$ be a finite rank subgroup of $\overline{F_0}^*$. Then the system of equations:
		\begin{align}
     P(Z_1,Z_2,Z_3)&=0\label{eq:system P eq1}\\
     xZ_1+yZ_2+Z_3&=0\label{eq:system P eq2}
    \end{align}
	has only finitely many solutions
	$(x,y,[Z_1:Z_2:Z_3])$ with
	$x,y\in G$, $[Z_1:Z_2:Z_3]\in \bP^2(F_0)$, and
	$Z_1Z_2Z_3\neq 0$.	
	\end{proposition}
	\begin{proof}
	If $g\geq 2$ then by Faltings' theorem \cite{Fal91_DA,Fal94_TG} (also see \cite[Chapter~6]{Lan83_FO}), 
	there are only finitely many
	$[z_1:z_2:z_3]\in\bP^2(F_0)$ such that
	$P(z_1,z_2,z_3)=0$. For such a $[z_1:z_2:z_3]$ with
	$z_1z_2z_3\neq 0$, the equation $xz_1+yz_2+z_3=0$
	has only finitely many solutions 
	$(x,y)\in G\times G$
	(see, for instance, \cite[Chapter~5]{BG06_HI}).
	
	Now assume $g=1$. Let $\cE$ denote the elliptic curve 
	defined by $P(Z_1,Z_2,Z_3)=0$ after choosing a point
	$O_{\cE}\in \cE(F_0)$ as the identity; we may assume
	$\cE(F_0)\neq \emptyset$ since the proposition is vacuously true otherwise. Let
	$\Gamma:=G\times G\times \cE(F_0)$
	which is a finite rank subgroup of the semi-abelian variety $S:=\Gm\times \Gm\times \cE$ \cite[Chapter~6]{Lan83_FO}. Let $(x,y)$ denote the coordinates of $\Gm\times \Gm$ and let $V$
	be the subvariety of $S$ defined by the equation
	$xZ_1+yZ_2+Z_3=0$. 
	We are now studying the set
	$V\cap \Gamma$. Since $\dim(V)=2$ and $V$ is not a translate
	of an algebraic subgroup, by the Mordell-Lang conjecture proved by Faltings \cite{Fal91_DA,Fal94_TG},
	McQuillan \cite{McQ95_DP}, and Vojta \cite{Voj96_IP}, we have that $V\cap \Gamma$ is the union of a finite set and finitely many sets of the form
	$(\gamma+C)\cap \Gamma$
	where $\gamma\in\Gamma$, $C$ is an
	algebraic subgroup of $S$ with $\dim(C)=1$,
	and $\gamma+C\subset V$. 
%	Write $\gamma=(3^{\alpha},3^{\beta},[\tilde{z}_1:\tilde{z}_2:\tilde{z}_3])$. 
	
	Assume that $\gamma+C$ is a translate of 
	an algebraic subgroup satisfying the above
	properties. 
	If the map $C\rightarrow \cE$ is non-constant then 
	$C$ has genus $1$ and, hence the map
	$C\rightarrow \Gm\times \Gm$ is constant since
	there cannot be a non-trivial algebraic group homomorphism
	from $C$ to $\Gm$. 
	Consequently $\gamma+C$ has the form
	$\{(\gamma_1,\gamma_2)\}\times\cE$ where
	$(\gamma_1,\gamma_2)\in \Gm\times\Gm$. Since
	$\gamma+C\subset V$, we have that 
	$\gamma_1Z_1+\gamma_2Z_2+Z_3=0$
	for every $[Z_1:Z_2:Z_3]\in \cE$, contradiction.
	Therefore the map $C\rightarrow \cE$ must be constant,
	in other words $C$ has the form $C_1\times \{O_{\cE}\}$
	where $C_1$ is an algebraic subgroup 
	of $\Gm\times\Gm$ with $\dim(C_1)=1$. Write $\gamma=(\gamma_x,\gamma_y,\gamma_{\cE})$
	with $(\gamma_x,\gamma_y)\in G\times G$ and $\gamma_{\cE}=:[\tilde{z}_1:\tilde{z}_2:\tilde{z}_3]\in \cE(F_0)$. Since
	$\gamma+C\subset V$, the translate of $C_1$
	by $(\gamma_x,\gamma_y)$ is given by the equation
	$\tilde{z}_1 x + \tilde{z}_2 y+\tilde{z}_3=0$. Equivalently, the algebraic group $C_1$ is given by the equation
	$\gamma_x^{-1}\tilde{z}_1 x+\gamma_y^{-1}\tilde{z}_2 y+\tilde{z}_3=0$. 
	This is
	possible only when $\tilde{z}_1\tilde{z}_2\tilde{z}_3=0$ and we finish the proof.
	\end{proof} 
	 
\begin{example}\label{eg:system P}
	Consider the system of equations
	\begin{align}
     Z_1^3+Z_2^3+9Z_3^3=0\label{eq:eg eq1}\\
     3^aZ_1+3^bZ_2+Z_3=0\label{eq:eg eq2}
    \end{align}
    with $a,b\in\Z$ and $[Z_1:Z_2:Z_3]\in \bP^2(F_0)$. Proposition~\ref{prop:finite} implies
    that there are only finitely many solutions
    outside the set
    $\{(m,m,[1:-1:0]):\ m\in\Z\}$. Later on, when $F_0=\Q$, we 
    will show that there does not exist 
    any solution satisfying $a>b\geq 0$
    confirming another conjecture of 
    Brassil-Reichstein \cite[Conjecture~14.3]{BR17_TH}.
\end{example}	 
	 
Let $n\geq 2$ be an integer, we recall the definition
of ``the general field extension" $E_n/F_n$ of degree $n$ 
over the base field $F_0$
from \cite[pp.~2]{BR17_TH}. Set $L_n:=F_0(x_1,\ldots,x_n)$, $F_n=L_n^{S_n}$, and $E_n:=L_n^{S_{n-1}}=F_n(x_1)$
where $x_1,\ldots,x_n$ are independent variables,
$S_n$ acts on $L_n$ by permuting $x_1,\ldots,x_n$ and 
$S_{n-1}$ acts on $L_n$ by permuting
$x_2,\ldots,x_n$. Theorem~\ref{thm:main 2} follows
from:
\begin{theorem}\label{thm:precise 2}
There is a finite subset $\cS$ of $\N\times\N$ depending
only on $F_0$ such that for every integer $n$
of the form $3^{k_1}+3^{k_2}+3^{k_3}$ with
integers $k_1>k_2>k_3\geq 0$
and $(k_1-k_3,k_2-k_3)\notin \cS$,
the following holds. 
 For every finite extension
$F'/F_n$ of degree prime to $3$, there does not
exist $0\neq \delta\in E':=F'\tensor_{F_n} E_n$
such that $\Tr_{E'/F'}(\delta)=\Tr_{E'/F'}(\delta^3)=0$.
In particular, there does not exist $0\neq \delta\in E_n$
such that $\Tr_{E_n/F_n}(\delta)=\Tr_{E_n/F_n}(\delta^3)=0$.
\end{theorem}
\begin{proof}
From \cite[Theorem~1.4]{BR17_TH}, \cite[Remark~11.3]{BR17_TH}, and put $a_1=k_1-k_3$ and $a_2=k_2-k_3$, it suffices to prove that the system of
equations 
	\begin{align}
     3^{a_1}Z_1^3+3^{a_2}Z_2^3+Z_3^3=0\label{eq:thm2 eq1}\\
     3^{a_1}Z_1+3^{a_2}Z_2+Z_3=0\label{eq:thm2 eq2}
    \end{align}
has only finitely many solutions
$(a_1,a_2,[Z_1:Z_2:Z_3])$ where
$[Z_1:Z_2:Z_3]\in \bP^2(F_0)$
and $a_1>a_2> 0$ are integers.

Write $a_i=3q_i+r_i$ with $q_i\in \Z$ and $r_i\in\{0,1,2\}$ for
$i=1,2$. It suffices
to show that for every \emph{fixed} pair
$(r_1,r_2)\in \{0,1,2\}^2$,
the system of equations
\begin{align}
     3^{r_1}Z_1^3+3^{r_2}Z_2^3+Z_3^3=0\label{eq:system q eq1}\\
     9^{q_1}Z_1+9^{q_2}Z_2+Z_3=0\label{eq:system q eq2}
\end{align}
has only finitely many solutions $(q_1,q_2,[Z_1:Z_2:Z_3])$
where $[Z_1:Z_2:Z_3]\in \bP^2(F_0)$, 
$q_1$ and $q_2$ are integers, and 
$3q_1+r_1>3q_2+r_2>0$. This last condition implies
$q_1>q_2\geq 0$. 

By Proposition~\ref{prop:finite}, it remains to consider
solutions satisfying $Z_1Z_2Z_3=0$. If $Z_3=0$, we have
$-(Z_2/Z_1)^3=3^{r_1-r_2}$,  
$-Z_2/Z_1=9^{q_1-q_2}$, and hence
$6\leq 6(q_1-q_2)=r_1-r_2$, contradiction.
Similarly, if $Z_2=0$, we have
$6\leq 6q_1=r_1$, contradiction. Finally, if
$Z_1=0$, we have $6q_2=r_2$ which implies
$q_2=r_2=0$ (otherwise $6\leq 6q_2=r_2$), contradicting 
the condition $3q_2+r_2>0$. This finishes the proof.
\end{proof}	
	
	\section{Proof of Theorem~\ref{thm:main 1}}
Throughout this section, let $F_0=\Q$. Let
$E_n/F_n$ be the general field extension of degree $n$
over $F_0=\Q$ as in the previous section. Theorem~\ref{thm:main 1} follows from:
\begin{theorem}\label{thm:precise 1}
For every $n$ of the form $3^{k_1}+3^{k_2}+3^{k_3}$
with integers $k_1>k_2>k_3\geq 0$ and for every 
finite extension $F'/F_n$ of degree prime to $3$,
there does not exist $0\neq \delta\in E':=F'\tensor_{F_n}E_n$
such that $\Tr_{E'/F'}(\delta)=\Tr_{E'/F'}(\delta^3)=0$.
In particular, there does not exist $0\neq \delta\in E_n$ such that $\Tr_{E_n/F_n}(\delta)=\Tr_{E_n/F_n}(\delta^3)=0$.
\end{theorem}

As explained in \cite[Chapter~14]{BR17_TH}, Theorem~\ref{thm:precise 1} follows from another conjecture
of Brassil-Reichstein \cite[Conjecture~14.3]{BR17_TH}:
\begin{conjecture}[Brassil-Reichstein]\label{conj:system BR}
The system of equations
\begin{align}
     Z_1^3+Z_2^3+9Z_3^3=0\label{eq:BR eq1}\\
     3^aZ_1+3^bZ_2+Z_3=0\label{eq:BR eq2}
\end{align}
    has no solution $(a,b,[Z_1:Z_2:Z_3])$
    where $a>b\geq 0$ are integers
    and  $[Z_1:Z_2:Z_3]\in \bP^2(\Q)$.
\end{conjecture}

In Example~\ref{eg:system P}, we explain why there are only finitely many solutions
$(a,b,[Z_1:Z_2:Z_3])$. This follows from Proposition~\ref{prop:finite} which uses the Mordell-Lang conjecture proved by Faltings, McQuillan, and Vojta. On the other hand, to prove that there is no solution, we need a different method using effective estimates. In fact, we establish a slightly stronger result than the statement
of Conjecture~\ref{conj:system BR}:
\begin{theorem}\label{thm:new system}
The only solution $(w,b,[Z_1:Z_2:Z_3])$ of the system
\begin{align}
     Z_1^3+Z_2^3+9Z_3^3=0\label{eq:new eq1}\\
     wZ_1+3^bZ_2+Z_3=0\label{eq:new eq2}
\end{align}
    with $w,b\in\Z$, $b\geq 0$, $3^{b+1}\mid w$, 
    and  $[Z_1:Z_2:Z_3]\in \bP^2(\Q)$ is $(0,0,[2:1:1])$.
\end{theorem}

We now spend the rest of this paper proving
Theorem~\ref{thm:new system}. From \eqref{eq:new eq1},
we cannot have $Z_1Z_2=0$. If $Z_3=0$ then
$Z_1/Z_2=-1$ and \eqref{eq:new eq2} gives
$w=3^b$ violating the condition
$3^{b+1}\mid w$. Let $(\tilde{w},\tilde{b},[\tilde{z}_1:\tilde{z}_2:\tilde{z}_3])$ be a solution,
and we can assume that $\tilde{z}_1$, $\tilde{z}_2$, and $\tilde{z}_3$ are 
non-zero integers with
$\gcd(\tilde{z}_1,\tilde{z}_2,\tilde{z}_3)=1$.

From $\gcd(\tilde{z}_1,\tilde{z}_2,\tilde{z}_3)=1$, we have $3\nmid \tilde{z}_1\tilde{z}_2$ and $-\tilde{z}_3=3^b \tilde{z}_4$ for some integer
$\tilde{z}_4$ with $3\nmid \tilde{z}_4$. Hence we have 
$\tilde{z}_1^3\mid 3^{3b+2}\tilde{z}_4^3-\tilde{z}_2^3$ and
$\tilde{z}_1\mid \tilde{z}_4-\tilde{z}_2$. This implies
\begin{equation}\label{eq:z1 divides 3^3b+2-1}
\tilde{z}_1\mid 3^{3b+2}-1.
\end{equation}
We now have:
\begin{equation}\label{eq:upper bound 1}
\vert \tilde{z}_2^3+9\tilde{z}_3^3\vert =\vert \tilde{z}_1^3\vert <3^{9b+6}.
\end{equation}
A result of Bennett \cite[Theorem~6.1]{Ben97_EM} gives:
\begin{equation}\label{eq:Mike}
\vert \tilde{z}_2^3+9\tilde{z}_3^3\vert \geq \frac{1}{3}\max\{\vert \tilde{z}_2\vert,\vert 3\tilde{z}_3\vert\}^{0.24}.
\end{equation}
Combining \eqref{eq:upper bound 1} and \eqref{eq:Mike}, we have:
\begin{equation}\label{eq:upper bound 2}
\max\{\vert \tilde{z}_2\vert,\vert 3\tilde{z}_3\vert\}<3^{37.5b+30}.
\end{equation}

This is our first step. Our next step is to give a lower bound for a quantity
that is closely related to 
$\max\{\vert \tilde{z}_2\vert,\vert 3\tilde{z}_3\vert\}$
and such a lower bound is much larger than 
$3^{37.5b+30}$
when $b$ is large. This will yield a strong upper bound on $b$.

Since $\tilde{z}_1^2-\tilde{z}_1\tilde{z}_2+\tilde{z}_2^2=(\tilde{z}_1+\tilde{z}_2)^2-3\tilde{z}_1\tilde{z}_2$ we have
that 
$\gcd (\tilde{z}_1+\tilde{z}_2,\tilde{z}_1^2-\tilde{z}_1\tilde{z}_2+\tilde{z}_2^2)\in\{1,3\}$ depending on
whether $3$ divides $\tilde{z}_1+\tilde{z}_2$. Moreover, if
$3\mid \tilde{z}_1+\tilde{z}_2$, then $9\nmid \tilde{z}_1^2-\tilde{z}_1\tilde{z}_2+\tilde{z}_2^2$. Therefore, \eqref{eq:new eq1}
gives:
\begin{equation}\label{eq:alpha and beta}
\tilde{z}_1+\tilde{z}_2=3^{3b+1}\alpha^3,\ \tilde{z}_1^2-\tilde{z}_1\tilde{z}_2+\tilde{z}_2^2=3\beta^3,\ \alpha\beta=\tilde{z}_4,\ 3\nmid \alpha\beta,\ \gcd(\alpha,\beta)=1.
\end{equation}

We wish to write the cubic curve given by 
equation~\eqref{eq:new eq1} 
into the standard Weierstrass form $y^2=x^3+Ax+B$. We have:
\begin{align}\label{eq:U and V}
\begin{split}
\frac{1}{4}(Z_1+Z_2)^3+\frac{3}{4}(Z_1+Z_2)(Z_1-Z_2)^2&=-9Z_3^3\\
\frac{1}{4}+\frac{3}{4}V^2&= 9U^3\\
V^2&=12U^3-\frac{1}{3}
\end{split}
\end{align}
with $U=\displaystyle\frac{-Z_3}{Z_1+Z_2}$ and $V=\displaystyle\frac{Z_1-Z_2}{Z_1+Z_2}$. Overall, we have:
\begin{equation}\label{eq:elliptic curve}
y^2=x^3-48,\ x=12U=\frac{-12Z_3}{Z_1+Z_2},\ y=12V=\frac{12(Z_1-Z_2)}{Z_1+Z_2}.
\end{equation}

Let $\cE$ be the elliptic curve given by the equation $y^2=x^3-48$. 
By a result of Selmer \cite[p.~357]{Sel51_TD}
as noted in \cite[Section~14]{BR17_TH}, we have that 
$\cE(\Q)$ is cyclic and generated by 
the point $G=(4,4)$. For every $P\in \cE(\Qbar)$, let 
$x(P)$ denote its $x$-coordinate. 

By \eqref{eq:U and V} and \eqref{eq:elliptic curve},
the solution $(\tilde{w},\tilde{b},[\tilde{z}_1:\tilde{z}_2:\tilde{z}_3])$ gives the point $(\tilde{x},\tilde{y})\in \cE(\Q)$
with
\begin{equation}\label{eq:tilde x and tilde y}
\tilde{x}=\frac{-12\tilde{z}_3}{\tilde{z}_1+\tilde{z}_2}= \frac{12\cdot 3^b \alpha\beta}{3^{3b+1}\alpha^3}=\frac{4\beta}{3^{2b}\alpha^2}.
\end{equation}
Let $N\geq 1$ such that $\tilde{x}=x([N] G)$. Let $\vert\cdot\vert_3$ denote the $3$-adic absolute value on $\Q$. 
By inspecting
the powers of $3$ that appear in the denominator
of $x(G)$, $x([2]G)$,$\ldots$ 
we observe that $N$ can be bounded below due to
$\vert \tilde{x}\vert_3=3^{2b}$. Indeed, we have the following:
\begin{proposition}\label{prop:bound N}
For $n\in\N$, write $n=3^m\ell$ with $\gcd(n,\ell)=1$, then
we have:
$$\vert x([n]G)\vert_3=3^{2m}.$$
\end{proposition}
\begin{proof}
We have $G=(4,4)$, $[2]G=(28,-148)$, and $[3]G=(73/9,595/27)$. 

\textbf{Claim 1:} assume that $P=[k]G$ for some $k\geq 1$
and $k\neq 3$. 
 If $\vert x(P)\vert_3=1$
then $\vert x(P+[3]G)\vert_3=1$.

Proof of Claim 1: write $P=(x_P,y_P)$. Since $\vert x_P\vert_3=1$
and $y_P^2=x_P^3-48$, we have $\vert y_P\vert_3=1$. Let
$$\lambda=\frac{y_P-\frac{595}{27}}{x_P-\frac{73}{9}},\ 
\nu=\frac{\frac{595}{27}x_P-\frac{73}{9}y_P}{x_P-\frac{73}{9}}.$$ 
From \cite[pp.~54]{Sil09_TA}, the $x$-coordinate of $P+[3]G$ is:
$$\lambda^2-\frac{73}{9}-x_P=\frac{-x_P^3+\frac{73}{9}x_P^2+\frac{5329}{81}x_P+y_P^2-\frac{1190}{27}y_P-48}{(x_P-\frac{73}{9})^2}.$$
This proves Claim 1 since
$$\left\vert -x_P^3+\frac{73}{9}x_P^2+\frac{5329}{81}x_P+y_P^2-\frac{1190}{27}y_P-48\right\vert_3 =\left\vert (x_P-\frac{73}{9})^2\right\vert_3=81.
$$

By induction, Claim 1 shows that $\vert x([n]G)\vert_3=1$ if $3\nmid n$. By induction again, it remains to prove the
following.

\textbf{Claim 2:} assume that $P=[k]G$ with $k\geq 1$. If
$\vert x(P)\vert_3\geq 1$ then
$\vert x([3]P)\vert_3= 9\vert x(P)\vert_3$.

Proof of Claim 2: write $P=(x_P,y_P)$. From \cite[pp.~105--106]{Sil09_TA}, consider:
$$\psi_3=3x^4-576x=3x(x^3-192)$$
$$\psi_2=2y,\ \psi_4=2y(2x^6-1920x^3-192^2),$$
$$\psi_2\psi_4=4y^2(2x^6-1920x^3-192^2)=4(x^3-48)(2x^6-1920x^3-192^2),$$
$$\phi_3=x\psi_3^2-\psi_2\psi_4=x^9+4608x^6+110592x^3-7077888,$$
$$f(x)=\frac{\phi_3}{\psi_3^2}=\frac{x^9+4608x^6+110592x^3-7077888}{9x^2(x^3-192)^2}$$
so that $x([3]P)=f(x_P)$. This proves Claim 2
since
$$\vert x_P^9+4608x_P^6+110592x_P^3-7077888\vert_3=\vert x_P^9\vert_3\ 
\text{and}\ 
\vert 9x_P^2(x_P^3-192)^2\vert_3=\frac{1}{9}\vert x_P^8\vert_3.$$
\end{proof}

Let $h$ denote the absolute logarithmic Weil height on
$\bP^1(\Qbar)$ and let $\hhat$ denote the N\'eron-Tate canonical
height on $\cE(\Qbar)$, see \cite[Chapter~8]{Sil09_TA}. We have $\Delta=-3^5\times 2^{12}$ 
and $j=0$. Then a result of Silverman \cite[pp.~726]{Sil90_TD} gives:
\begin{equation}\label{eq:Silverman}
-2.13<\hhat(P)-\frac{1}{2}h(x(P))<2.222
\end{equation}

We calculate the point $[25] G$ explicitly, then apply \eqref{eq:Silverman} for this point, and use
$\hhat([25] G)=625\hhat(G)$ to obtain:
\begin{equation}\label{eq:height G}
0.25<\hhat(G).
\end{equation}

From \eqref{eq:Silverman} and \eqref{eq:height G}, we have:
\begin{equation}\label{eq:lower bound 1}
h(\tilde{x})>2\hhat([N]G)-4.444>0.5 N^2-4.444.
\end{equation}
From \eqref{eq:tilde x and tilde y} and \eqref{eq:lower bound 1}, we have:
\begin{equation}\label{eq:upper bound 3}
\frac{1}{3^{b+1}\alpha}\max\{\vert 12z_3\vert,\vert z_1+z_2\vert\}
=\max\{\vert 4\beta\vert,\vert 3^{2b}\alpha^2\vert\}
\geq e^{h(\tilde{x})}>e^{0.5N^2-4.444}.
\end{equation}

From \eqref{eq:z1 divides 3^3b+2-1} and \eqref{eq:upper bound 2},
we have:
\begin{equation}\label{eq:upper bound 4}
\max\{\vert z_1+z_2\vert, \vert 12z_3\vert\}<3^{37.5b+31.5}.
\end{equation}

Equations \eqref{eq:upper bound 3} and \eqref{eq:upper bound 4}
give:
\begin{equation}\label{eq:N^2}
0.5N^2-4.444<(36.5b+30.5)\ln(3).
\end{equation}

Proposition~\ref{prop:bound N} together with
$\vert \tilde{x}\vert_3=3^{2b}$
imply $3^b\mid N$. Together
with \eqref{eq:upper bound 2}, we have:
\begin{equation}\label{eq:final}
3^{2b}\leq N^2<81b+76
\end{equation}

Hence $b<3$. We check the following cases:
\begin{itemize}
\item [(i)] $b=0$. So $z_1\mid 8$ and $N^2<76$ which gives $N\in\{1,\ldots,8\}$.

\item [(ii)] $b=1$. So $z_1\mid 242$, $3\mid N$ and $N^2<157$
which give $N\in\{3,6,9,12\}$.

\item [(iii)] $b=2$. So $z_1\mid 6560$, $9\mid N$
and $N^2<238$ which give $N=9$.
\end{itemize}

We consider all the values of $N$ and $b$ below. Since we may
replace $(z_1,z_2,z_3)$ by $(-z_1,-z_2,-z_3)$, we always choose
$\alpha>0$.	The pair $(\alpha,\beta)$ is determined
using $x([N]G)=\displaystyle\frac{4\beta}{3^{2b}\alpha^2}$, $3\nmid \alpha\beta$, and
$\gcd(\alpha,\beta)=1$. 

The case $N=1$ and $b=0$
gives $x(G)=4=\displaystyle\frac{4\beta}{\alpha^2}$,
hence $\alpha=\beta=1$, $\tilde{z}_1+\tilde{z}_2=3$, $\tilde{z}_1^2-\tilde{z}_1\tilde{z}_2+\tilde{z}_2^2=3$, 
$\tilde{z}_1\mid 8$. Overall, we have the solution $(0,0,[2:1:1])$. 

For other values of $(N,b)$, from \eqref{eq:z1 divides 3^3b+2-1} and \eqref{eq:alpha and beta}, we have:
\begin{equation}
\vert \tilde{z}_1\vert < 3^{2b+2}\ \text{and}\ 
\vert \tilde{z}_2\vert < 3^{3b+1}\vert \alpha^3\vert +  3^{2b+2}.
\end{equation}
Then using $\tilde{z}_1\tilde{z}_2=\displaystyle\frac{1}{3}\left ((\tilde{z}_1+\tilde{z}_2)^2-(\tilde{z}_1-\tilde{z}_1\tilde{z}_2+\tilde{z}_2^2)\right)=3^{6b+1}\alpha^6-\beta^3$, we have:
\begin{equation}\label{eq:the last one}
3^{2b+2}(3^{3b+1}\vert \alpha^3\vert + 3^{2b+2})> \vert 3^{6b+1}\alpha^6-\beta^3\vert.
\end{equation}
We can readily check that \eqref{eq:the last one} fails
for the data in the following table and this finishes the proof.

\begin{center}
\begin{tabular}{|c|c|c|c|}
\hline
$(N,b)$ & $x([N]G)$ & $\alpha$ & $\beta$\\  \hline
& & & \\
$(2,0)$ & $28$& $1$ & $7$ \\ 
& & & \\
$(3,0)$ & $\displaystyle\frac{73}{9}$ & $6$ & $73$\\ 
& & & \\
$(3,1)$ & $\displaystyle\frac{73}{9}$ & $2$ & $73$\\ 
& & & \\
$(4,0)$ & $\displaystyle\frac{9772}{1369}$ & $37$ & $2443$\\ 
& & & \\
$(5,0)$ & $\displaystyle\frac{1184884}{32041}$ & $179$ & $296221$\\ 
& & & \\
$(6,0)$ & $\displaystyle\frac{48833569}{12744900}$ & $7140$ & $48833569$\\ 
& & & \\
$(6,1)$ & $\displaystyle\frac{48833569}{12744900}$ & $2380$ & $48833569$\\ 
& & & \\
$(7,0)$ & $\displaystyle\frac{238335887764}{143736121}$ & $11989$ & $59583971941$\\ 
& & & \\
$(8,0)$ & $\displaystyle\frac{292913655316492}{69305008951369}$ & $8324963$ & $73228413829123$\\ 
& & & \\
$(9,1)$ & $\displaystyle\frac{587359987541570953}{26773203784287249}$ & $109083462$ & $587359\ldots$\\ 
& & & \\
$(9,2)$ & $\displaystyle\frac{587359987541570953}{26773203784287249}$ & $36361154$ &  $587359\ldots$\\ 
& & & \\
$(12,1)$ & 
$\displaystyle\frac{44507186275594022064781897173121}{871004453785806995703095216400}$ & $622184\ldots$ &  $445071\ldots$\\ 
& & & \\ \hline
\end{tabular}
\end{center}

	\bibliographystyle{amsalpha}
	\bibliography{Hermite_Joubert_4} 	

\end{document}